\documentclass[11pt]{scrartcl}
\usepackage[T1]{fontenc}
\usepackage{lmodern}
\usepackage{amsthm, amsfonts, amssymb,latexsym,amsmath}
\usepackage[a4paper, total={6in, 8in}]{geometry}
\usepackage[dvipsnames,svgnames,table]{xcolor}
\usepackage[colorlinks=true,linkcolor=RoyalBlue,urlcolor=RoyalBlue,citecolor=PineGreen]{hyperref}
\usepackage{tikz,pgf}
\usepackage{tikz-cd}
\usetikzlibrary{calc}
\usepackage[nameinlink]{cleveref}
\usepackage{booktabs}
\usepackage{enumerate,enumitem}

\title{Computing the number of realisations of a rigid graph}
\date{}
\usepackage{authblk}

\author[1]{Sean Dewar}
\author[2]{Georg Grasegger}
\author[2,3]{Josef Schicho}
\author[2]{Ayush Kumar Tewari}
\author[2]{Audie Warren}
\affil[1]{School of Mathematics, University of Bristol}
\affil[2]{Johann Radon Institute for Computational and Applied Mathematics, Austrian Academy of Sciences}
\affil[3]{Johannes Kepler University Linz, Research Institute for Symbolic Computation}

\colorlet{colbg}{white}
\colorlet{colfg}{black}
\colorlet{colgraphv}{colfg!75!white}
\colorlet{colgraphe}{colfg!55!white}
\colorlet{colG}{DarkSeaGreen}
\definecolor{colR}{HTML}{CC6677}
\definecolor{colO}{HTML}{DDCC77}
\definecolor{colB}{HTML}{6699CC}

\colorlet{col1}{colB}
\colorlet{col2}{colR}
\colorlet{col3}{colG}
\colorlet{col4}{colO}
\colorlet{cola}{colfg!75!colbg}

\newcommand{\R}{\mathbb{R}}
\newcommand{\C}{\mathbb{C}}

\newcommand{\ccountfuncd}[1]{c_{#1}}

\newcommand{\ccountd}[2]{\ccountfuncd{#1}(#2)}
\newcommand{\ccount}[1]{\ccountd{2}{#1}}

\newcommand{\scountfuncd}[1]{c^{\circ}_{#1}}
\newcommand{\scountfunc}[1]{\scountfuncd{2}}
\newcommand{\scountd}[2]{\scountfuncd{#1}(#2)}
\newcommand{\scount}[1]{\scountd{2}{#1}}

\newcommand{\maxsymb}{\mathbf{M}}
\newcommand{\maxccount}[2]{\maxsymb_2^{#1}(#2)}
\newcommand{\maxscount}[2]{\maxsymb_2^{\circ,#1}(#2)}

\newcommand{\modop}{\,\mathrm{mod}\,}

\newcommand{\pyrigi}{\textsc{PyRigi}}

\pgfkeys{/pgf/number format/.cd,
    int detect,
    /pgf/number format/set thousands separator={\,}
}

\newtheorem{theorem}{Theorem}
\newtheorem{proposition}[theorem]{Proposition}
\newtheorem{lemma}[theorem]{Lemma}

\theoremstyle{remark}
\newtheorem{remark}[theorem]{Remark}
\newtheorem{example}[theorem]{Example}

\Crefname{example}{Example}{Examples}
\Crefname{remark}{Remark}{Remarks}
\Crefname{lemma}{Lemma}{Lemmas}
\Crefname{proposition}{Proposition}{Propositions}

\tikzstyle{vertex}=[fill=colgraphv,circle,inner sep=0pt, minimum size=4pt]
\tikzstyle{hvertex}=[vertex,minimum size=6pt,fill=colR]
\tikzstyle{edge}=[line width=1.5pt,colgraphe]
\tikzstyle{labelsty}=[font=\scriptsize]
\tikzstyle{cdarrow}=[-{Classical TikZ Rightarrow}]

\tikzstyle{aline}=[draw=cola]
\tikzstyle{bline}=[draw=cola!10!colbg]
\tikzstyle{alabelsty}=[cola,font=\scriptsize]
\tikzstyle{legend}=[minimum size=4pt]
\tikzstyle{dvertex}=[fill=colgraphv,circle,inner sep=0pt,line width=0.1pt]

\begin{document}
\maketitle

\begin{abstract}
    A graph is said to be \emph{rigid} if, given a generic realisation of the graph as a bar-and-joint framework in the plane, there exist only finitely many other realisations of the graph with the same edge lengths modulo rotations, reflections and translations.
    In recent years there has been an increase of interest in determining exactly what this finite amount is, hereon known as the \emph{realisation number}.
    Combinatorial algorithms for the realisation number were previously known for the special cases of minimally rigid and redundantly rigid graphs.
    In this paper we provide a combinatorial algorithm to compute the realisation number of any rigid graph, and thus solve an open problem of Jackson and Owen.
    We then adapt our algorithm to compute: (i) spherical realisation numbers, and (ii) the number of rank-3 PSD matrix completions of a generic partial matrix.
\end{abstract}

\section{Introduction}
A (simple undirected) graph $G$ is called \emph{rigid in $\R^2$} (or \emph{rigid} for short) if, for a generic%
\footnote{%
    In this paper, a \emph{generic} point of a variety (in our case, a complex irreducible algebraic set) with defining equations $f_1,\ldots,f_m$ is a point whose set of coordinates is algebraically independent over the field of coefficients of $f_1,\ldots,f_m$.
    For points in either $\mathbb{R}^n$ or $\mathbb{C}^n$, this condition can be simplified to the coordinates being algebraically independent over $\mathbb{Q}$.%
}
realisation $r \in (\R^2)^{|V|}$, there are only finitely many realisations $r'$, up to rotations, reflections and translations, such that for any edge $\{i,j\}\in E(G)$, 
the distance between the positions of the vertices $i,j$ in the realisation $r$ is the same as their distance in $r'$.
The well-known theorem of Pollaczek-Geiringer/Laman \cite{Geiringer1927,Laman1970} characterises the \emph{minimally rigid graphs} (i.e., those rigid graphs with no proper rigid subgraph sharing the same vertex set) with at least two vertices: they are exactly the sparse graphs fulfilling the equation $|E(G)|=2|V(G)|-3$.
Here a graph $G$ is sparse (i.e.\ $(2,3)$-sparse) if the inequality $|E(H)|\le 2|V(H)|-3$ is true for all subgraphs $H$ of $G$ with at least two vertices. 
Consequently, rigid graphs with at least two vertices can be characterised as graphs that contain a spanning subgraph that satisfies the Pollaczek-Geiringer/Laman condition.

For rigid graphs, one may ask for the number of realisations, up to rotations, reflections and translations, which have the same edge distances. For real realisations (that is, realisations in the real plane) this number depends on the choice of the generic realisation $r$ (see, for example, \cite[Fig.~1 \& 2]{JacksonOwen2019}).
However, if we pass to complex realisations, and replace the Euclidean distance by the square of its extension to the complex numbers,
and replace the group of rotations and translations by the complexification of the algebraic group $\mathbb E_2$ generated by translations, reflections, and orthogonal linear transformations,
then the number of realisations depends only on the graph -- not on the generic realisation chosen.
We name this number $\ccount{G}$ -- the realisation number of $G$. 
Note that the realisation number is defined in different ways in the literature, depending on whether reflection is considered or not.
In our case we do `mod out' reflections.
Hence, the triangle graph has realisation number 1.
That means that our notation is consistent for instance with \cite{JacksonOwen2019,DewarGrasegger2024}, while for example \cite{CapcoGalletEtAl2018,GalletGraseggerSchicho2020} yield realisation numbers that are a multiple of 2 with respect to the realisation number considered in this paper.

For minimally rigid graphs, the paper \cite{CapcoGalletEtAl2018} gives an algorithm that computes this realisation number.
Additionally, Jackson and Jord\'{a}n \cite{JacksonJordan2005} gave a characterisation of graphs that are \emph{globally rigid}, that is, those graphs where $\ccount{G}=1$.
They show that a graph $G$ with more than 3 vertices is globally rigid if and only if it has the following two properties.
Firstly, it must be \emph{redundantly rigid}, meaning that if you remove any edge the graph remains rigid. Secondly,
it must be \emph{3-connected}, meaning that if you remove any two vertices the graph remains connected.

In this paper, we give a formula to compute $\ccount{G}$ for any given rigid graph $G$, thus solving Problem 8.1 in \cite{JacksonOwen2019}.
The case where the graph $G$ is redundantly rigid but not 3-connected was already solved by Jackson and Owen \cite{JacksonOwen2019};
our contribution is a formula for the realisation number of $G$ when $G$ is not redundantly rigid.

\subsection{A rigorous definition of realisation numbers}

Before we begin, we first give a more rigorous definition of our earlier described concepts.
For any graph $G = (V,E)$ we define the \emph{measurement map}
\begin{align*}
    f_G: \mathbb C^{2|V|} &\longrightarrow \mathbb C^{|E|},\\
    \big(x_1,y_1,\ldots,x_{|V|},y_{|V|}\big) &\longmapsto \big((x_i-x_j)^2+(y_i-y_j)^2\big)_{\{i,j\} \in E}.
\end{align*}
The measurement map simply evaluates all the squared edge lengths of a given realisation of $G$. If $V$ is a finite set containing at least two elements and $K_V$ is the complete graph with vertex set $V$, we make the following definition.
\begin{equation*}
    R_V := \overline{f_{K_{V}}(\mathbb C^{2|V|})},
\end{equation*}
where the closure is the Zariski closure. The variety $R_V$ contains all possible squared distance vectors between $|V|$ points in $\mathbb C^2$, and is well-known as the (complex 2-dimensional) Cayley-Menger variety.
Let us define $E_2$ to be the complexification of the Euclidean group $\mathbb E_2$. Any squared distance vector lying in $R_V$ can be considered to be an orbit of $E_2$ within $\mathbb C^{2|V|}$.
Indeed, the variety $R_V$ has dimension $2|V|-3$, corresponding to the quotienting of $\mathbb{C}^{2|V|}$ by the 3-dimensional complex Lie group $E_2$.

Using the measurement map, we now define $c_2(G)$ to be the number of points in the set $f_G^{-1}(f_G(\rho))/E_2$ for any generic point $\rho \in \mathbb C^{2|V|}$
(See \cite{JacksonOwen2019} for a proof that $c_2(G)$ is well-defined).
With this, a graph $G$ is rigid if and only if $c_2(G)$ is finite.

\subsection{Statement of result}\label{sec:mainresults}

As stated, our main result is an equation for the realisation number of a rigid graph $G$ which is not redundantly rigid.
This equation involves realisation numbers of graphs on fewer number of vertices than $G$, allowing recursive applications.

Suppose that we would like to calculate the realisation number of the graph $G$. If $G$ is globally rigid, then we know that its realisation number is $1$ and we are done.
If not, then by the classification of globally rigid graphs by Jackson and Jordan \cite{JacksonJordan2005}, either $G$ fails to be $3$-connected, or fails to be redundantly rigid (or possibly both).
Consider the case when $G$ is not redundantly rigid. If this happens, there exists some edge $e$ such that $G - e$ is not rigid. Then $G-e$ has a decomposition into maximal rigid subgraphs $G_1,\ldots,G_m$.
Let $H_1,\ldots,H_m$ be corresponding  minimally rigid spanning subgraphs; that is, each $H_i$ is a spanning minimally rigid subgraph of $G_i$.
We then define $H:= H_1 \cup \ldots \cup H_m \cup \{e\}$, and for the moment claim that $H$ is minimally rigid (this is proved later).
With this terminology set, we can state our results.
We include the case of $G$ not being 3-connected by Jackson and Owen \cite{JacksonOwen2019} for completeness.

\begin{theorem}\label{thm:main}
    Let $G$ be a graph which is rigid but not globally rigid. Then one of the following two cases hold.
    \begin{enumerate}[label=(\roman*)]
        \item\label{it:main:notred} $G$ is not redundantly rigid: In this case, let $e$ be such that $G-e$ is not rigid, 
        let $G_1,\ldots,G_m$ be the maximal rigid subgraphs of $G-e$,
        and let $H_i$ be a minimally rigid spanning subgraph of $G_i$ for each $i \in \{1,\ldots,m\}$.
        Then the graph $H := H_1 \cup \cdots \cup H_m \cup \{e\}$ is minimally rigid,
        and we have
        \[ \ccount{G} = \ccount{H} \prod_{i=1}^m \frac{\ccount{G_i}}{\ccount{H_i}} .\]
        \item\label{it:main:notcon} $G$ is not 3-connected: In this case, let $K,L$ be induced subgraphs of $G$ such that $V(K) \cup V(L) = V(G)$, $V(K) \cap V(L) = \{u,v\}$, and $E(K) \cup E(L) = E(G)$.
        Then, given $s = \{u,v\}$, we have
        \begin{equation*}
            \ccount{G} =
            \begin{cases}
                2\ccount{K} \ccount{L+s} &\text{if $s \notin E(G)$, $K$ is rigid, $L$ is not rigid}, \\
                2\ccount{K+s} \ccount{L+s} &\text{if $s \in E(G)$ or both $K$ and $L$ are rigid}.
            \end{cases}
        \end{equation*}
    \end{enumerate}
\end{theorem}

Given \Cref{thm:main}, we can calculate the realisation number for any rigid graph. Indeed, let $G$ be any rigid graph. If $G$ is minimally rigid, then by the existing algorithm of \cite{CapcoGalletEtAl2018}, its realisation number can be computed. If $G$ is globally rigid, its realisation number is $1$.
If neither of these occur, we apply \Cref{thm:main}; we can write $\ccount{G}$ in terms of realisation numbers of graphs on strictly fewer vertices.
The algorithm then proceeds recursively; if any of the smaller graphs are minimally rigid or globally rigid, we can compute their realisation number.
If not, we apply \Cref{thm:main} again.
This process clearly concludes at some point, since we have a finite number of vertices in $G$ and the number of vertices in the graphs considered strictly decreases at each stage.

It is natural to ask how often \Cref{thm:main} is needed when computing realisation numbers.
\Cref{tab:graph_numbers} shows how many graphs on few vertices are minimally/globally/redundantly rigid. What is particularly interesting is the last column of \Cref{tab:graph_numbers},
which shows the number of rigid graphs which are neither minimally rigid nor redundantly rigid and which do not have a 2-cut,
i.\,e., those graphs that require \Cref{thm:main}\ref{it:main:notred} (the more computationally involved of the two cases) to compute their realisation number.

\begin{table}[ht]
    \centering
    \begin{tabular}{rrrrrrrr}
         \toprule
               &         &           &          &             &         & not min. rigid \\[-5pt]
               &         & minimally & globally & redundantly &         & not red. rigid \\[-5pt]
         $|V|$ & rigid   & rigid     & rigid    & rigid       & 2-cut   & 3-con          \\\midrule
         6     & \pgfmathprintnumber{42}      & \pgfmathprintnumber{13}        & \pgfmathprintnumber{15}       & \pgfmathprintnumber{17}          & \pgfmathprintnumber{25}      & \pgfmathprintnumber{0}              \\
         7     & \pgfmathprintnumber{377}     & \pgfmathprintnumber{70}        & \pgfmathprintnumber{132}      & \pgfmathprintnumber{142}         & \pgfmathprintnumber{241}     & \pgfmathprintnumber{1}              \\
         8     & \pgfmathprintnumber{6199}    & \pgfmathprintnumber{608}       & \pgfmathprintnumber{2346}     & \pgfmathprintnumber{2496}        & \pgfmathprintnumber{3815}    & \pgfmathprintnumber{14}             \\
         9     & \pgfmathprintnumber{180878}  & \pgfmathprintnumber{7222}      & \pgfmathprintnumber{80433}    & \pgfmathprintnumber{83046}       & \pgfmathprintnumber{100009}  & \pgfmathprintnumber{234}            \\
         10    & \pgfmathprintnumber{9464501} & \pgfmathprintnumber{110132}    & \pgfmathprintnumber{5105493}  & \pgfmathprintnumber{5180419}     & \pgfmathprintnumber{4350705} & \pgfmathprintnumber{5765}           \\\bottomrule
    \end{tabular}
    \caption{Number of graphs with different properties for small number of vertices. }
    \label{tab:graph_numbers}
\end{table}

\subsection{Examples}

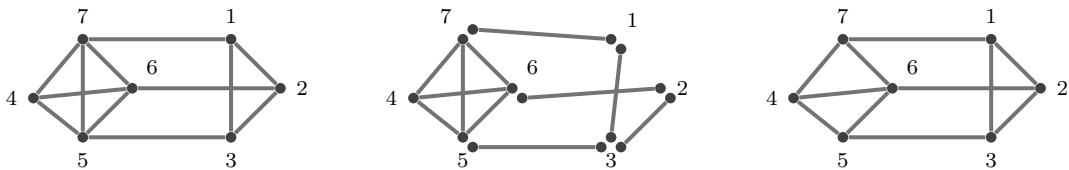
\begin{figure}[t]
    \centering
    \begin{tikzpicture}[scale=1.3] 
        \node[vertex,label={[labelsty]above:1}] (1) at (1.5,1) {};
        \node[vertex,label={[labelsty]right:2}] (2) at (2,0.5) {};
        \node[vertex,label={[labelsty]below:3}] (3) at (1.5,0) {};
        \node[vertex,label={[labelsty]left:4}] (4) at (-0.5,0.4) {};
        \node[vertex,label={[labelsty]below:5}] (5) at (0,0) {};
        \node[vertex,label={[labelsty]45:6}] (6) at (0.5,0.5) {};
        \node[vertex,label={[labelsty]above:7}] (7) at (0,1) {};
        \draw[edge] (1)--(2) (1)--(3) (1)--(7) (2)--(3) (2)--(6) (3)--(5) (4)--(5) (4)--(6) (4)--(7) (5)--(6) (5)--(7) (6)--(7);
    \end{tikzpicture}\qquad
    \begin{tikzpicture}[scale=1.3]
        \node[vertex,label={[labelsty]25:1}] (1a) at (1.5,1) {};
        \node[vertex] (1b) at ($(1a)+(0.1,-0.1)$) {};
        \node[vertex,label={[labelsty]right:2}] (2a) at (2,0.5) {};
        \node[vertex] (2b) at ($(2a)+(0.1,-0.1)$) {};
        \node[vertex,label={[labelsty]below:3}] (3a) at (1.5,0) {};
        \node[vertex] (3b) at ($(3a)+(0.1,-0.1)$) {};
        \node[vertex] (3c) at ($(3a)+(-0.1,-0.1)$) {};
        \node[vertex,label={[labelsty]left:4}] (4) at (-0.5,0.4) {};
        \node[vertex,label={[labelsty]below:5}] (5a) at (0,0) {};
        \node[vertex] (5b) at ($(5a)+(0.1,-0.1)$) {};
        \node[vertex,label={[labelsty]45:6}] (6a) at (0.5,0.5) {};
        \node[vertex] (6b) at ($(6a)+(0.1,-0.1)$) {};
        \node[vertex,label={[labelsty]100:7}] (7a) at (0,1) {};
        \node[vertex] (7b) at ($(7a)+(0.1,0.1)$) {};
        \draw[edge] (1b)--(3a) (1a)--(7b) (2b)--(3b) (2a)--(6b) (3c)--(5b) (4)--(5a) (4)--(6a) (4)--(7a) (5)--(6a) (5)--(7a) (6a)--(7a);
    \end{tikzpicture}\qquad
    \begin{tikzpicture}[scale=1.3] 
        \node[vertex,label={[labelsty]above:1}] (1) at (1.5,1) {};
        \node[vertex,label={[labelsty]right:2}] (2) at (2,0.5) {};
        \node[vertex,label={[labelsty]below:3}] (3) at (1.5,0) {};
        \node[vertex,label={[labelsty]left:4}] (4) at (-0.5,0.4) {};
        \node[vertex,label={[labelsty]below:5}] (5) at (0,0) {};
        \node[vertex,label={[labelsty]45:6}] (6) at (0.5,0.5) {};
        \node[vertex,label={[labelsty]above:7}] (7) at (0,1) {};
        \draw[edge] (1)--(2) (1)--(3) (1)--(7) (2)--(3) (2)--(6) (3)--(5) (4)--(5) (4)--(6) (4)--(7) (5)--(6) (6)--(7);
    \end{tikzpicture}
    \caption{(Left): The rigid graph $G$ given in \Cref{ex:plane}. (Middle): The maximally rigid subgraphs of $G' = G - \{1,2\}$. (Right): The minimally rigid subgraph $H$ of $G$.}
    \label{fig:nonred}
\end{figure}

\begin{example}\label{ex:plane}
    Let us consider the graph $G$ from \Cref{fig:nonred}~(left).
    This graph is clearly not minimally rigid since it has one too many edges.
    It is not redundantly rigid either and hence not globally rigid, since the deletion of the edge $\{1,2\}$ would result in a non-rigid graph $G'$.
    
    We split the graph $G'$ into its maximal rigid components $G_1,\ldots,G_6$ as shown in \Cref{fig:nonred} (middle).
    Let $G_1$ be the graph with vertices $4,5,6,7$. It is the only subgraph that is not minimally rigid.
    We compute $\ccount{G_1}=1$ since it is globally rigid.
    Each other subgraph $G_i$ is a single edge, and hence the respective $H_i$ are also single edges.
    This in turn implies $\ccount{G_i}=\ccount{H_i}=1$ for each $i \neq 1$.
    It is not so hard to convince oneself that the graph $H_1$ obtained from $G_1$ by deleting the edge $\{5,7\}$ is minimally rigid with $\ccount{H_1}=2$.
    The algorithm given in \cite{CapcoGalletEtAl2018} applied to the graph $H=H_1\cup\ldots\cup H_6\cup\{1,2\}$ gives that $\ccount{H}=24$.
    By applying \Cref{thm:main}\ref{it:main:notred} we get $\ccount{G}=\ccount{H} \prod_{i=1}^m \frac{\ccount{G_i}}{\ccount{H_i}}=24\cdot \frac{1}{2}=12$.
\end{example}

\begin{example}\label{ex:plane2}
    Now consider the graph $\widetilde{G}$ from \Cref{fig:nonred2} (left).
    Again, this graph is neither minimally rigid nor is it redundantly rigid.
    The algorithm described in \cite{CapcoGalletEtAl2018} gives that the graph $\widetilde{H} = \widetilde{G} -\{5,7\}$ has a realisation number of 672.
    By combining \Cref{thm:main}\ref{it:main:notred} ($e=\{10,11\}$) with the observation that the subgraph generated on the first 7 vertices is the graph $G$ from \Cref{ex:plane},
    we have that $\ccount{\widetilde{G}}=\ccount{\widetilde{H}} \cdot \frac{\ccount{G}}{\ccount{H}}=672\cdot \frac{12}{24}=336$.
\end{example}

\begin{figure}[t]
    \centering
    \begin{tikzpicture}[scale=0.8,baseline={(0,0)}]
        \node[vertex,label={[labelsty]60:6}] (1) at (0.2,3) {};
        \node[vertex,label={[labelsty]right:5}] (2) at (2,2) {};
        \node[vertex,label={[labelsty]left:7}] (3) at (-2,2) {};
        \node[vertex,label={[labelsty]below:4}] (4) at (0,1.5) {};
        \node[vertex,label={[labelsty]right:8}] (5) at (0.7,0.5) {};
        \node[vertex,label={[labelsty]left:9}] (6) at (-0.7,0.5) {};
        \node[vertex,label={[labelsty]right:10}] (7) at (1.2,0) {};
        \node[vertex,label={[labelsty]left:11}] (8) at (-1.2,0) {};
        
        \node[vertex,label={[labelsty]above:2}] (a) at (0,6) {};
        \node[vertex,label={[labelsty]right:3}] (b) at (2,5) {};
        \node[vertex,label={[labelsty]left:1}] (c) at (-2,5) {};
        
        \draw[edge] (a)--(b) (a)--(c) (b)--(c) (a)--(1) (b)--(2) (c)--(3);            
        \draw[edge] (1)--(2) (1)--(3) (1)--(4) (2)--(3) (2)--(4) (2)--(7) (3)--(4) (3)--(8) (4)--(5) (4)--(6) (5)--(6) (5)--(7) (6)--(8) (7)--(8);
    \end{tikzpicture}\qquad
    \begin{tikzpicture}[scale=0.8,baseline={(0,0)}]
        \node[vertex,label={[labelsty]60:6}] (1) at (0.2,3) {};
        \node[vertex,label={[labelsty]right:5}] (2a) at (2,2) {};
        \node[vertex,label={[labelsty]left:7}] (3a) at (-2,2) {};
        \node[vertex] (4a) at (0,1.5) {};            
        \node[vertex] (2b) at (2.3,1.7) {};
        \node[vertex,label={[labelsty]right:10}] (7b) at (1.5,-0.3) {};            
        \node[vertex] (3b) at (-2.3,1.7) {};
        \node[vertex,label={[labelsty]left:11}] (8b) at (-1.5,-0.3) {};            
        \node[vertex,label={[labelsty]below:4}] (4b) at (0,1.2) {};
        \node[vertex,label={[labelsty]below:8}] (5b) at (0.7,0.2) {};
        \node[vertex,label={[labelsty]below:9}] (6b) at (-0.7,0.2) {};            
        \node[vertex] (5a) at (0.9,0.5) {};
        \node[vertex] (7a) at (1.4,0) {};            
        \node[vertex] (6a) at (-0.9,0.5) {};
        \node[vertex] (8a) at (-1.4,0) {};
        
        \node[vertex,label={[labelsty]above:2}] (a) at (0,6) {};
        \node[vertex,label={[labelsty]right:3}] (b) at (2,5) {};
        \node[vertex,label={[labelsty]left:1}] (c) at (-2,5) {};
        
        \draw[edge] (a)--(b) (a)--(c) (b)--(c) (a)--(1) (b)--(2a) (c)--(3a);            
        \draw[edge] (1)--(2a) (1)--(3a) (1)--(4a) (2a)--(3a) (2a)--(4a) (3a)--(4a);            
        \draw[edge] (2b)--(7b) (3b)--(8b);            
        \draw[edge] (4b)--(5b) (4b)--(6b) (5b)--(6b);            
        \draw[edge] (5a)--(7a) (6a)--(8a);
    \end{tikzpicture}\qquad
    \begin{tikzpicture}[scale=0.8,baseline={(0,0)}]
        \node[vertex,label={[labelsty]60:6}] (1) at (0.2,3) {};
        \node[vertex,label={[labelsty]right:5}] (2) at (2,2) {};
        \node[vertex,label={[labelsty]left:7}] (3) at (-2,2) {};
        \node[vertex,label={[labelsty]below:4}] (4) at (0,1.5) {};
        \node[vertex,label={[labelsty]right:8}] (5) at (0.7,0.5) {};
        \node[vertex,label={[labelsty]left:9}] (6) at (-0.7,0.5) {};
        \node[vertex,label={[labelsty]right:10}] (7) at (1.2,0) {};
        \node[vertex,label={[labelsty]left:11}] (8) at (-1.2,0) {};
        
        \node[vertex,label={[labelsty]above:2}] (a) at (0,6) {};
        \node[vertex,label={[labelsty]right:3}] (b) at (2,5) {};
        \node[vertex,label={[labelsty]left:1}] (c) at (-2,5) {};
        
        \draw[edge] (a)--(b) (a)--(c) (b)--(c) (a)--(1) (b)--(2) (c)--(3);            
        \draw[edge] (1)--(2) (1)--(3) (1)--(4) (2)--(4) (2)--(7) (3)--(4) (3)--(8) (4)--(5) (4)--(6) (5)--(6) (5)--(7) (6)--(8) (7)--(8);
    \end{tikzpicture}
    \caption{(Left): The rigid graph $\widetilde{G}$ given in \Cref{ex:plane2}. (Middle): The maximally rigid subgraphs of $\widetilde{G} - \{10,11\}$. (Right): The minimally rigid subgraph $\widetilde{H}$ of $\widetilde{G}$.}\label{fig:nonred2}
\end{figure}
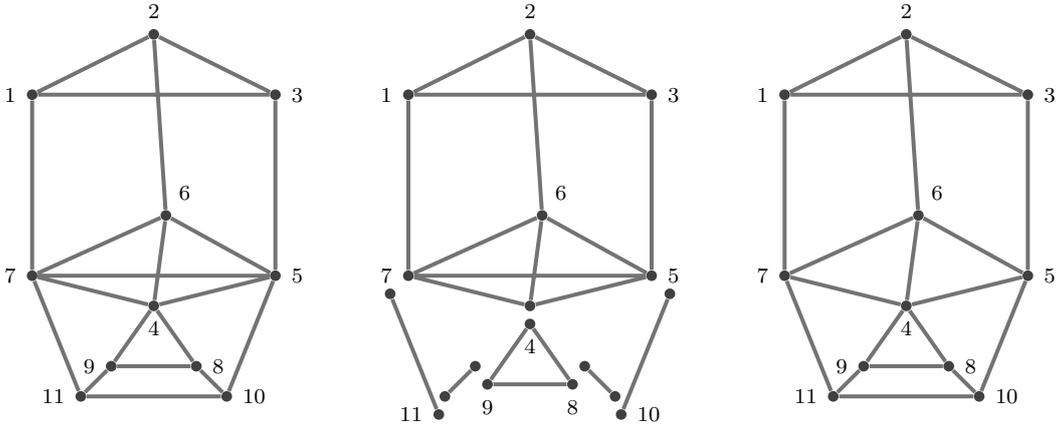

\begin{remark}
    We note here that an alternative approach for computing $\ccount{G}$ in \Cref{ex:plane} would be to use \cite[Thm.~6.9]{JacksonOwen2019}.
    This gives that $\ccount{G}=12\ccount{G_1}\ccount{G''}=12$, where $G''$ is the triangle graph formed on vertices $\{1,2,3\}$.
    However, there is no such previously known technique that can compute $\ccount{\widetilde{G}}$ in \Cref{ex:plane2}.
\end{remark}

\subsection{Structure of paper and notation}

This paper is structured as follows.
We cover all the required background material for rigidity theory in \Cref{sec:rigiditytheory}.
In \Cref{sec:proof_of_main}, we provide a proof of \Cref{thm:main}.
This algorithm is then adapted in \Cref{sec:sphere} to compute spherical realisation numbers and rank-3 PSD matrix completions.
We conclude the paper in \Cref{sec:compute} with a range of computational results.

We use the following graph theory notation throughout the paper.
All graphs we consider are simple and undirected and have at least two vertices. A subgraph of $G$ is a graph $H$ such that 
$V(H)\subseteq V(G)$ and $E(H)\subseteq E(G)$. Subgraphs may be  induced, meaning $E(H)$ is the subset of all edges 
with vertices in $V(H)$, or spanning, meaning $V(G)=V(H)$. 
If $e\in E(G)$, then $G-e$ is defined as the subgraph of $G$ with edge set $E(G)\setminus\{e\}$ and vertex set $V(G)$.
If $\{u,v\}\subset V(G)$ is a subset of cardinality 2, then $G+\{u,v\}$ is defined as the graph with vertex set $V(G)$ and edge set $E(G)\cup \{\{u,v\} \}$. We often abuse notation and denote by $e$ both an edge of a graph and the subgraph consisting of the single edge $e$.

\section{Background on rigidity theory}\label{sec:rigiditytheory}

In this section we cover the required background material on rigidity theory from two different perspectives; firstly using combinatorics, and then using algebraic geometry.

\subsection{Combinatorial rigidity}

We define the \emph{rank} of a graph $G$ to be the non-negative integer value
\begin{equation*}
    r(G) := \max \big\{ |E(H)| : H \text{ is a sparse subgraph of } G  \big\}.
\end{equation*}
A consequence of the Pollaczek-Geiringer/Laman condition for rigidity is that a graph $G$ with at least two vertices is rigid if and only if $r(G) = 2|V(G)|-3$.
In fact, the map $r_G:F \mapsto r((V(G),F))$ for $F \subseteq E(G)$ defines a rank function (in the matroidal sense) on the edge set $E(G)$;
with this, we say that the matroid $(E(G),r_G)$ is the \emph{rigidity matroid of $G$}.
We refer any reader eager to know more on rigidity matroids (and matroids in general) to \cite{GraverServatiusServatius1993}.

Through the language of rigidity matroids, many concepts in matroid theory have direct analogues in rigidity theory.
For this paper, we focus on two in particular:
\begin{itemize}
    \item Given an edge set $F \subseteq E(G)$ and a spanning subgraph $H \subset G$,
    we say that $F$ is \emph{in the span of} $H$ if $r(H+F) = r(H)$.
    It can be shown that if $F$ is in the span of $H$ then $F$ is in the span of $H'$ for any $H \subseteq H' \subseteq G$ (see, for example, \cite[Lemma 1.4.3]{Oxley2011}).
    \item A graph $G$ with two or more vertices is a \emph{circuit} if it is rigid, $|E(G)| = 2|V(G)|-2$ and $r(G-e) = r(G)$ for all $e \in E(G)$;
    as a consequence of the Pollaczek-Geiringer/Laman condition, this is equivalent to $E(G)$ being a circuit in the rigidity matroid of $G$.
\end{itemize}
Circuits also lead to an alternative definition for minimal rigidity; specifically, a graph is minimally rigid if and only if it is rigid and contains no circuit subgraphs.

Using these combinatorial concepts, we prove the following key lemma.

\begin{lemma} \label{lem:Hismr}
    Let $G$ be a rigid graph.
    Suppose that there exists an edge $e\in E(G)$ such that $G-e$ is not rigid. 
    Let $G_1,\dots,G_m$ be the maximal
    rigid subgraphs of $G-e$. 
    For $i=1,\dots,m$, let $H_i$ be a minimally rigid spanning subgraph of $G_i$, and let $H$ be the spanning subgraph of $G$ with edge set $\bigcup_{i=1}^m E(H_i)\cup\{e\}$.
    Then $H$ is minimally rigid.
\end{lemma}

\begin{proof}
    To show that $H$ is minimally rigid, it suffices to show that it rigid and contains no circuits.
    We first show that $H$ is rigid.
    Since $E(G_i) \setminus E(H_i)$ is in the span of $H_i$ (as $r(H_i) = r(G_i)$),
    the set $E(G-e) \setminus E(H-e)$ is in the span of $H-e$.
    Hence, $r(H-e) = r(G-e)$.
    Since $G$ is rigid but $G-e$ is not, we have 
    \begin{equation*}
        r(G) = 2|V(G)|-3, \qquad r(H-e) = r(G-e) = 2|V(G)|-4.
    \end{equation*}
    This implies that $e$ is not in the span of $G-e$,
    and so $e$ is also not in the span of $H-e$.
    Thus, $r(H) = r(H-e)+1 = 2|V(G)|-3$, and $H$ is rigid.

    Now suppose for contradiction that $H$ contains a circuit $H'$.
    As $e$ is not in the span of $H-e$,
    then $e$ is not contained in the span $H'$;
    in particular, $H' \subseteq H-e$.
    As all circuits are rigid,
    it follows that $H'$ is contained in some maximal rigid subgraph $G_j$ of $G-e$.
    However, it now follows that $H'$ is contained in $H_j$,
    contradicting that $H_j$ is sparse.
\end{proof}

We also require the following characterisation of the rank of a graph.

\begin{lemma}[{\cite[Lemma 4.2]{Jordan2010}}]\label{lem:ly}
    Let $G_1,\ldots,G_m$ be the maximal rigid subgraphs of a connected graph $G$ with at least two vertices. Then
    \begin{equation*}
        r(G) := \sum_{i=1}^m \big( 2|V(G_i)| - 3 \big).
    \end{equation*}
\end{lemma}

\subsection{Algebraic geometry and rigidity}

We first recall that a morphism $f:X \rightarrow Y$ of varieties is called \emph{dominant} if $f(X)$ is Zariski dense in $Y$, and we call $f$ \emph{generically finite} if the fibre $f^{-1}(y):= \{ x \in X : f(x) = y\}$ is a finite set for generic $y \in Y$.
The \emph{degree} of a dominant and generically finite map is then the number of points in $f^{-1}(y)$ for any generic $y \in Y$.
An important property for degrees and fiber dimension is that they often behave multiplicatively and additively with respect to composition: 
specifically, if $f:X \rightarrow Y$ and $g:Y \rightarrow Z$ are dominant then 
\begin{equation*}
    \dim ((g \circ f)^{-1}(\gamma)) = \dim g^{-1}(\alpha) + \dim f^{-1}(\beta) \qquad \text{ for generic } \beta \in Y, ~ \alpha, \gamma \in Z,
\end{equation*}
and if $f,g$ are also generically finite then $\deg(g \circ f) = \deg(g) \deg(f)$.

We now wish to view our realisation numbers as degrees of dominant maps of varieties.
We note that different authors use various definitions of a `realisation variety', using techniques such as pinning vertices (e.g., \cite{JacksonOwen2019,DewarGrasegger2024}).
In this paper we do not wish to use pinning, as it makes the upcoming \emph{decomposition map} (\Cref{sec:proof_of_main}) difficult to define.
Instead, we want to exploit the following useful property of the variety $R_V$: by projecting away coordinates we can reach $R_{V'}$ for any $V' \subseteq V$.

We construct our dominant maps for computing realisation numbers as follows.
Let $G = (V,E)$ be a graph. 
We set the \emph{image variety} $I_G \subset \mathbb{C}^{|E|}$ to be the Zariski closure of the projection of $R_V$ onto the entries in $E$.
Since each image variety is also the image of the corresponding measurement map,
the variety $I_G$ has dimension $r(G)$, and is irreducible, as the closure of the image of an irreducible variety.
From this, we define the \emph{graph map}
\begin{equation*}
    p_G: R_V \rightarrow I_G,
\end{equation*}
which projects a vector in $R_V$ to the entries given by the edges in $E$. 
Note that, by definition, each graph map is dominant.
Moreover, it follows from the factorisation $f_G = p_G \circ f_{K_{V(G)}}$ that a graph is rigid if and only if its graph map is generically finite.
Furthermore, if $G$ is minimally rigid then we have $I_G = \mathbb C^{|E|}$, since in this case $\dim(I_G) = r(G) = 2|V| - 3 = |E|$.

If $G$ is rigid, then the number of connected components of $f_G^{-1}(f_G(\rho))$ for generic $\rho$ is $2c_2(G)$; see, for example, \cite[Lemma 5]{LubbesMakhulEtAl2025}.
Similarly, if $G$ is rigid then $\deg(p_G) = c_2(G)$.
A proof of this fact can be found in \Cref{app}.

\section[Proof of Main Theorem]{Proof of \Cref{thm:main}}\label{sec:proof_of_main}

As stated previously, \Cref{thm:main}\ref{it:main:notcon} was originally proven by Jackson and Owen \cite[Theorem 6.6]{JacksonOwen2019}.
We now proceed with \Cref{thm:main}\ref{it:main:notred}.
The main idea of the proof is to factorise the graph map $p_G$ in a suitable way, using the following map. 

We define an \emph{edge decomposition} of a graph $G$ as a set of connected subgraphs $S=\{G_1,\dots,G_m\}$ such that $E(G)$ is the
disjoint union of $E(G_1),\dots,E(G_m)$, each $G_i$ contains at least one edge, and any pair $(G_i, G_j)$ have at most one vertex in common.
We can then define the \emph{decomposition map} 
\[
    d_S:R_{V(G)}\to\prod_{i=1}^m R_{V(G_i)}
\] 
given by restricting a distance vector to the vertex set of each subgraph.
For any connected graph blue with at least two vertices, the set of maximal rigid subgraphs is an edge decomposition;
two maximal rigid subgraphs share at most one vertex and never share an edge, and any edge is rigid as a subgraph and is therefore contained in some maximal rigid subgraph. 

For this section we make use of the following lemma, which gives information about the decomposition map. 

\begin{lemma} \label{lem:fiberdegree}
    Let $H$ be a graph which is minimally rigid, and let $S=\{H_1,\dots,H_t\}$ be an edge decomposition of $H$ into minimally rigid subgraphs.
    Then the decomposition map $d_S:R_{V(H)}\to\prod_{i=1}^t R_{V(H_i)}$ is dominant and generically finite and has degree
    $\frac{\ccountd{2}{H}}{\ccountd{2}{H_1}\cdots \ccountd{2}{H_t}}$.
\end{lemma}

\begin{proof}
    The graph map $p_H:R_{V(H)}\to\C^{|E(H)|}$ (which is dominant and generically finite, since $H$ is minimally rigid) factors into the following commutative diagram.
    \begin{center}       
        \begin{tikzpicture}
            \node[] (r)at (0,0) {$R_{V(H)}$};
            \node[] (p) at (4,0) {$\prod_{i=1}^t R_{V(H_i)}$};
            \node[] (c) at (4,-2) {$\C^{|E(H)|}$};
            \draw[cdarrow] (r) to node[above,labelsty] {$d_S$} (p);
            \draw[cdarrow] (r) to node[below,labelsty] {$p_H$} (c);
            \draw[cdarrow] (p) to node[right,labelsty] {$p_S$} (c);
        \end{tikzpicture}
    \end{center}
    where $p_S$ is the product of graph maps of the subgraphs $H_i$
    \[
        p_S:=(p_{H_1},\dots,p_{H_t}):\prod_{i=1}^t R_{V(H_i)}\to \prod_{i=1}^t \C^{|E(H_i)|} = \C^{|E(H)|}.
    \]
    We first show that $d_S$ is a dominant map. Indeed, we have
    \begin{equation*}
        \dim\left(\prod_{i=1}^t R_{V(H_i)}\right)  = \sum_{i=1}^t \dim(R_{V(H_i)}) 
       =  \sum_{i=1}^t \left( 2|V(H_i)| - 3 \right) 
         = \sum_{i=1}^t |E(H_i)|
         = |E(H)|,
    \end{equation*}
    where the third equality is implied by each $H_i$ being minimally rigid, and the last equality is because $S$ is an edge decomposition.
    Therefore, all three varieties in the commutative diagram have the same dimension, and since $p_H$ is dominant, so is $d_S$.
    Therefore, all three maps in the diagram are dominant and generically finite. For generically finite and dominant maps, the degrees multiply when composed.
    We therefore have 
    \begin{equation*}
        \deg(p_H) = \deg(d_S)\deg(p_S) \implies \ccount{H} = \prod_{i=1}^t \ccount{H_i} \cdot \deg(d_S)
    \end{equation*}
    which proves the result.
\end{proof}

\begin{remark}
    The proof of \Cref{lem:fiberdegree} extends to higher dimensions, however the decomposition of $H$ into minimally rigid subgraphs sharing at most one vertex is no longer guaranteed.
    This is important for the dominance of the decomposition map $d_S$;
    if, in three dimensions, two of the subgraphs, say $H_1$ and $H_2$ of $G$ share a pair of vertices $v_1,v_2$, then after applying $d_S$ to any realisation $\rho$ of $G$, we must have that, for $d_S(\rho) = (\rho_1,\rho_2,\ldots,\rho_t)$, the equation $d(\rho_1(v_1),\rho_1(v_2)) = d(\rho_2(v_1),\rho_2(v_2))$ is always satisfied.
    This shows that $d_S$ cannot be a dominant map in such a situation. 
    The dominance of $d_S$, however, is crucial for the proof of \Cref{thm:main}. Hence, extending our results to higher dimensions is a subject for further research.
\end{remark}

We now prove a small lemma regarding the image varieties of an edge decomposition. 
Recall that $I_G := \overline{p_G(R_{V(G)})}$.

\begin{lemma}\label{imagevarietydecomp}
    Let $G$ be a rigid graph such that $G-e$ is not rigid for some edge $e \in E(G)$.
    Let $G_1,\ldots,G_m$ be the maximal rigid subgraphs of $G-e$. Then we have
    \begin{equation*}
        I_G = I_{G_1} \times \dots \times I_{G_m} \times \mathbb{C}.
    \end{equation*}
\end{lemma}

\begin{proof}
    The inclusion $I_G \subseteq I_{G_1} \times \ldots \times I_{G_m} \times \mathbb{C}$ is clear.
    It now suffices to show that both sets have the same dimension:
    since both $I_G$ and $I_{G_1} \times \ldots \times I_{G_m} \times \mathbb{C}$ are irreducible and closed in the Zariski topology, this immediately implies the sets are equal.
    Since $G$ is rigid, we have $\dim I_G= r(G) = 2|V(G)|-3$.
    By \Cref{lem:ly},
    we then have 
    \begin{equation*}
        \dim (I_{G_1} \times \ldots \times I_{G_m}) = \sum_{i=1}^m \dim I_{G_i} = \sum_{i=1}^m \big( 2|V(G_i)| - 3 \big) = r(G-e) = 2|V(G)|-4.
    \end{equation*}
    Hence, $I_{G_1} \times \ldots \times I_{G_m} \times \mathbb{C}$ has dimension $2|V(G)|-3$, as required. We note that the end factor of $\mathbb C$ should be thought of as the image variety of the single edge $e$.
\end{proof}

We are now ready to prove the first case of \Cref{thm:main}.

\begin{lemma} \label{lem:12}
    Let $G$ be a rigid graph such that $G-e$ is not rigid for some edge $e \in E(G)$. 
    Let $G_1,\ldots,G_m$ be the maximal rigid subgraphs of $G-e$,
    and let $H_i$ be a minimally rigid spanning subgraph of $G_i$ for each $i \in \{1,\ldots,m\}$.
    Then the graph $H := H_1 \cup \cdots \cup H_m \cup \{e\}$ is minimally rigid,
    and we have
    \[ 
        \ccount{G} = \ccount{H} \prod_{i=1}^m \frac{\ccount{G_i}}{\ccount{H_i}} .
    \]
\end{lemma}

\begin{proof}
    As in the proof of \Cref{lem:fiberdegree}, the graph map $p_G:R_{V(G)}\to I_G$ factors into the product of graph maps of subgraphs
    \[
        p_S:=(p_{G_1},\dots,p_{G_m},p_e):\prod_{i=1}^m R_{V(G_i)}\times R_e\to \prod_{i=1}^{m}I_{G_i}\times\C = I_G
    \] 
    and the
    decomposition map $d_S$; this can be seen in the following commutative diagram, where we note that by \Cref{imagevarietydecomp} we have $I_G=I_{G_1}\times\cdots\times I_{G_m}\times I_e$:
    \begin{center}
        \begin{tikzpicture}
            \node[] (r)at (0,0) {$R_{V(G)}$};
            \node[] (p) at (4,0) {$\prod_{i=1}^m R_{V(G_i)} \times R_e$};
            \node[] (i) at (4,-2) {$I_G$};
            \draw[cdarrow] (r) to node[above,labelsty] {$d_S$} (p);
            \draw[cdarrow] (r) to node[below,labelsty] {$p_G$} (i);
            \draw[cdarrow] (p) to node[right,labelsty] {$p_S$} (i);
        \end{tikzpicture}
    \end{center}
    The map $d_S$ is also the decomposition map for the decomposition $\{H_1,\dots,H_m,e\}$ of $H$, since $R_{V(G)} = R_{V(H)}$ and $R_{V(G_i)} = R_{V(H_i)}$ for all $i$. By \Cref{lem:Hismr}, $H$ is minimally rigid, and by \Cref{lem:fiberdegree}, $d_S$ is a generically finite and dominant map with degree $\frac{\ccount{H}}{\ccount{H_1}\cdots \ccount{H_m}}$.
    For the subgraph $e$, note that $R_e=I_e=\C$, and $p_e$ is the identity map. We then have
    \begin{equation*}
        \deg(p_G) = \deg(d_S)\cdot \deg(p_S) \implies \ccount{G} = \frac{\ccount{H}}{\ccount{H_1}\cdots\ccount{H_m}}\prod_{i=1}^m \ccount{G_i}
    \end{equation*}
    as needed (note that $p_e$ has degree 1 and is omitted from the formula).
\end{proof}

We can now finish the proof of our main result.

\begin{proof}[Proof of \Cref{thm:main}]
    Let $G$ be a graph that is rigid but not globally rigid.
    By \cite{Hendrickson1992},
    $G$ is either not redundantly rigid, or $G$ is not 3-connected.
    If $G$ is either not redundantly rigid, then there exists an edge $e \in E(G)$ such that $G-e$ is not rigid.
    Case \ref{it:main:notred} now follows by \Cref{lem:12}.
    If $G$ is not 3-connected, we see that Case \ref{it:main:notcon} holds by \cite[Theorem 6.6]{JacksonOwen2019}.
\end{proof}

\section{Spherical realisations and PSD matrix completions}\label{sec:sphere}

In this section we describe how our algorithm can be adapted to computing spherical realisation numbers, and an application to positive semi-definite matrix completion.

\subsection{Computing spherical realisation numbers}

Our proof of \Cref{thm:main}\ref{it:main:notred} also extends to counting realisations on the (complex) sphere $\mathbb S^2$.
Indeed, a graph is minimally rigid in $\mathbb C^2$ if and only if is also minimally rigid on the sphere (see for instance \cite{SaliolaWhiteley2007}).
Furthermore, a graph is globally rigid in $\mathbb C^2$ if and only if it is globally rigid in $\mathbb S^2$;
the real variant of this statement follows from a result of Connelly and Whiteley \cite[Theorem 12]{ConnellyWhiteley2010},
which can then be adapted to the complex setting using a combination of results from Gortler and Thurston \cite[Theorem 1]{GortlerThurston2014} and the first and second authors \cite[Theorem 1.2]{DewarGrasegger2024}.
Therefore, Jackson and Jord\'{a}n's characterisation of globally rigid graphs also holds in the sphere, and we have the same division into two cases. 

The relevant changes to our definitions are as follows; firstly, we would define a spherical measurement map for $G=(V,E)$ as follows:
\begin{equation*}
    f_G^{\mathbb S} : (\mathbb S^2)^{|V|} \rightarrow \mathbb C^{|E|}
\end{equation*}
which yields a new `realisation variety' given by $R_V^{\mathbb S} := \overline{f_{K_{V}}^{\mathbb S}((\mathbb S^2)^{|V|})}$.
The group $E_2$ would be changed to the (complex) orthogonal group  $O(3)$. The definition of an edge decomposition remains the same.
At this point the proofs continue in the same way as above;
any result we have used which was specific for the plane  (for instance the Pollaczek-Geiringer/Laman condition) also holds for spherical realisations, as noted above.

We denote the spherical realisation number by $\scountfunc\ $.
An algorithm to compute this number for minimally rigid graphs is provided in \cite{GalletGraseggerSchicho2020}.
It is known from \cite{DewarGrasegger2024} that $\ccount{G}\leq\scount{G}$ for all rigid graphs.

\begin{example}
    The graph from \Cref{ex:plane} has $\scount{G}=16$ since $\scount{H}=32$.
    This is the smallest rigid but not minimally rigid graph where $\ccount{G}\neq\scount{G}$.
\end{example}

\subsection{Counting rank-3 PSD matrix completions}

The spherical realisation number has an additional application with regards to positive semidefinite matrix completion.

To be more specific, set $\mathcal{S}_+^n(r)$ to be the variety of $n \times n$ real symmetric PSD matrices with rank at most $r$.
Given a graph $G=([n],E)$, we now define the projection
\begin{equation*}
    \pi_G : \mathcal{S}_+^n(r) \rightarrow \mathbb{R}^{|E|} \times \mathbb{R}^n, ~ M \mapsto \left( (M_{ij})_{\{i,j\} \in E} , (M_{kk})_{k \in [n]} \right).
\end{equation*}
Any vector $\lambda$ in the image of $\pi_G$ is said to be a \emph{rank-$r$ partial PSD matrix}, and any matrix in the fiber $\pi_G^{-1}(\lambda)$ is said to be a \emph{completion of $\lambda$}.
For more background on this problem and its links to rigidity theory, see \cite{SingerCucuringu2010,JacksonJordanTanigawa2016}.

Ideally, we wish to understand the number of completions for a generic rank-$r$ partial PSD matrix.
However, as is the case for counting real realisations, this number depends on the choice of generic realisation.
We can solve this issue by allowing for complex solutions.
We extend $\mathcal{S}_{+}^n(r)$ to all complex symmetric matrices that can be decomposed as $A^T A$ for some $r \times n$ complex matrix $A$;
the logic here being that a real symmetric matrix is PSD with rank at most $r$ if and only if such a decomposition exists with $A$ being real.
Any such complex matrix that is mapped to our chosen generic  rank-$r$ partial PSD matrix is said to be a \emph{complex completion}.
There is always exactly one complex completion of a rank-1 partial PSD matrix, and the number of complex completions of a rank-2 partial PSD matrix given by a connected graph with $k$ biconnected components is exactly $2^{k-1}$.

By combining the techniques developed by Singer and Cucuringu \cite{SingerCucuringu2010} and a result of the first and second author \cite[Theorem 1.2]{DewarGrasegger2024},
we see that the $(r-1)$-dimensional spherical realisation number of a graph is exactly the number of complex completions of any rank-$r$ partial PSD matrix.
Using this correspondence, we can extend our counting algorithm for rank-$r$ partial PSD matrices to the case of $r=3$.
Specifically, if $G$ is rigid, then $c_2^\circ(G)$ is exactly the number of complex completions of any generic rank-3 partial PSD matrix, and we apply the spherical variant of our counting algorithm to compute this number; otherwise, there are infinitely many complex completions.

\section{Computational results}\label{sec:compute}
In this section we collect results and statistics of computations of realisation numbers for graphs with a reasonably small number of vertices.
As such we have computed realisation numbers for all rigid graphs with less than 11 vertices both for the plane and the sphere. The results can be found in \cite{DatasetRigidRealisations}.
The computations have been done in python based on \pyrigi\ \cite{PyRigi} and the code will be made available via \pyrigi.

\subsection{Computational results for realisation numbers}

We focus now on rigid graphs with $2|V|-3+k$ edges and a high number of realisations.
For $k=0$, i.e.\ the minimally rigid case, this was done before in \cite{GraseggerKoutschanTsigaridas2020,Grasegger2025}.
Let $\maxccount{k}{n}$ be the maximal $\ccount{G}$ over all rigid graphs $G=(V,E)$ with $|E|=2|V|-3+k$.
We show in \Cref{tab:maxk} some values of $\maxccount{k}{n}$, which can be found by computing the realisation numbers of all rigid graphs with the respective number of vertices and edges.
Note that for $n\leq 10$ we have computed realisation numbers of all rigid graphs which satisfy the desired edge count. They can be found at \cite{DatasetRigidRealisations}.
For $n\geq 11$ we only computed realisations numbers for all rigid graphs with the respective number of edges and minimum degree 3, since adding a degree 2 vertex to a graph always doubles the number of realisations (e.g., \cite[Lemma 7.1]{DewarGrasegger2024}).
The two missing entries have not been computed yet, due to the large number of graphs.
For $k=0$ these numbers were computed in \cite{CapcoGalletEtAl2018}.
\begin{table}[ht]
    \centering
    \begin{tabular}{rrrrr}
         \toprule
         $n=|V|$ & $\maxccount{0}{n}$ & $\maxccount{1}{n}$ & $\maxccount{2}{n}$ & $\maxccount{3}{n}$\\\midrule
         6     &    12 &     4 &   2 &   2 \\
         7     &    28 &    12 &   4 &   4 \\
         8     &    68 &    28 &  12 &  12 \\
         9     &   172 &    72 &  28 &  28 \\
         10    &   440 &   172 &  80 &  72 \\
         11    &  1144 &   440 & 192 & 172 \\
         12    &  3090 &  1216 &   - &   - \\
         \bottomrule
    \end{tabular}
    \caption{The values of $\maxccount{k}{n}$ for $k\in\{0,1,2,3\}$ and $n\leq 11$. The graphs that obtain these numbers can be seen in \Cref{tab:max_cert}.}
    \label{tab:maxk}
\end{table}

Using a generalised fan construction as described in \cite{GraseggerKoutschanTsigaridas2020}, we can generate graphs of any size for which we can easily compute the realisation count.
The construction essentially glues several copies of a graph on a common subgraph.
For instance, when we glue copies of a rigid graph with $2|V|-2$ edges on a common $K_4$ subgraph,
the resulting graph again has $2|V|-2$ edges.
The realisation number is then just a power of the realisation number of the initial graph since $K_4$ is globally rigid.
From this we get
\begin{equation}\label{eq:bound_genfank1}
	\maxccount{1}{n} \geq 2^{(n-4)\modop(|V|-4)} \cdot
	\ccount{G}^{\!\lfloor(n-4)/(|V|-4)\rfloor} \qquad (n\geq4).
\end{equation}
Indeed, all the graphs attaining $\ccount{G}=\maxccount{1}{n}$ for $n\leq 12$ have a $K_4$ subgraph,
and hence we get the following bound.
\begin{theorem}
	\label{thm:lower-k1}
	The maximal number of realisations $\maxccount{1}{n}$, for $n\geq 4$, satisfies
	\begin{equation*}
		\maxccount{1}{n} \geq 2^{(n-4)\modop 8} \cdot 1216^{\!\lfloor(n-4)/8\rfloor}.
	\end{equation*}
	This means $\maxccount{1}{n}$ grows at least as $\bigl(\!\sqrt[8]{1216}\bigr)^n$, which is approximately $2.43006^n$.
\end{theorem}

Similarly, we can get bounds for any $\maxccount{k}{n}$ by gluing on a globally rigid subgraph with $2|V|-3+k$ edges.
Both $\maxccount{2}{n}$ and $\maxccount{3}{n}$ grow at least as $\left(\sqrt[6]{172}\right)^n$, which is approximately $2.35824^n$.
For $k=3$ this is obtained from the respective graph in \Cref{tab:max_cert} with 11 vertices.
For $k=2$ the graph from the table does not have a globally rigid subgraph with $2|V|-3+k$ edges, so we instead use another graph with $\ccount{G}=172$ (see \Cref{sec:cert}).

Note that these bounds have been found by exhaustive computation on graphs with few vertices, and can probably be improved by further computations on graphs with more vertices.
Although this is doable for individual graphs, the number of graphs on which to run the algorithm quickly becomes unfeasible; for example, there are already 891\,750\,296 rigid graphs on 11 vertices.

\subsection{Computational results for spherical realisation numbers}

We have also computed all spherical realisation numbers for rigid graphs with at most 10 vertices ( see \cite{DatasetRigidRealisations}).
Additionally for $n\geq 11$ we computed graphs with minimum degree 3 which suffices for attaining the maximum.
\Cref{tab:maxk_sphere} shows the maximal spherical realisation number $\maxscount{k}{n}$ for graphs with $n$ vertices and $2n-3+k$ edges. The realisation numbers for $k=0$ have been previously computed in \cite{GalletGraseggerSchicho2020,Grasegger2025}.
\begin{table}[ht]
    \centering
    \begin{tabular}{rrrrr}
         \toprule
         $n=|V|$ & $\maxscount{0}{n}$ & $\maxscount{1}{n}$ & $\maxscount{2}{n}$ & $\maxscount{3}{n}$\\\midrule
          6 &   16 &   4 &   2 &   2 \\
          7 &   32 &  16 &   4 &   4 \\
          8 &   96 &  32 &  16 &  16 \\
          9 &  288 & 128 &  32 &  32 \\
         10 &  768 & 320 & 128 & 128 \\
         11 & 2176 & 896 & 384 & 320 \\ 
         \bottomrule
    \end{tabular}
    \caption{The maximal number of spherical realisations for a given number of vertices $\maxscount{k}{n}$ for $k\in\{0,1,2,3\}$ and $n\leq 10$. Graph that obtain these numbers can be seen in \Cref{tab:max_cert_sphere}.}
    \label{tab:maxk_sphere}
\end{table}

All of the graphs from \Cref{tab:max_cert_sphere} have a suitable globally rigid subgraph on which to use a generalised fan construction.
This yields that $\maxscount{1}{n}$ grows at least as $(\sqrt[7]{896})^n$, which is approximately $2.64094^n$. 
Both $\maxscount{2}{n}$ and $\maxscount{3}{n}$ grow at least as $\left(\sqrt[5]{128}\right)^n$, which is approximately $2.63902^n$.
These bounds are obtained by the values in \Cref{tab:maxk_sphere}; in particular, the corresponding graphs with 10 vertices in \Cref{tab:max_cert_sphere}.
In this case the graphs with 11 vertices do not yield a better bound.
Again, further computations can improve this bound. 

Similarly to \cite[Section~7]{DewarGrasegger2024}, we compare the spherical realisation count to the planar count, but for all rigid graphs.
\Cref{fig:comparison} shows this comparison for graphs with 5--9 vertices.
That a significant proportion of all graphs have a realisation number of 1 is a consequence of the sharp threshold characterisation for Erd\H{o}s-Renyi random graph global rigidity \cite{JacksonServatiusServatius2007,LewNevoPeledRaz2023}.
The full data on rigid realisation numbers for up to 10 vertices is available at \cite{DatasetRigidRealisations}.
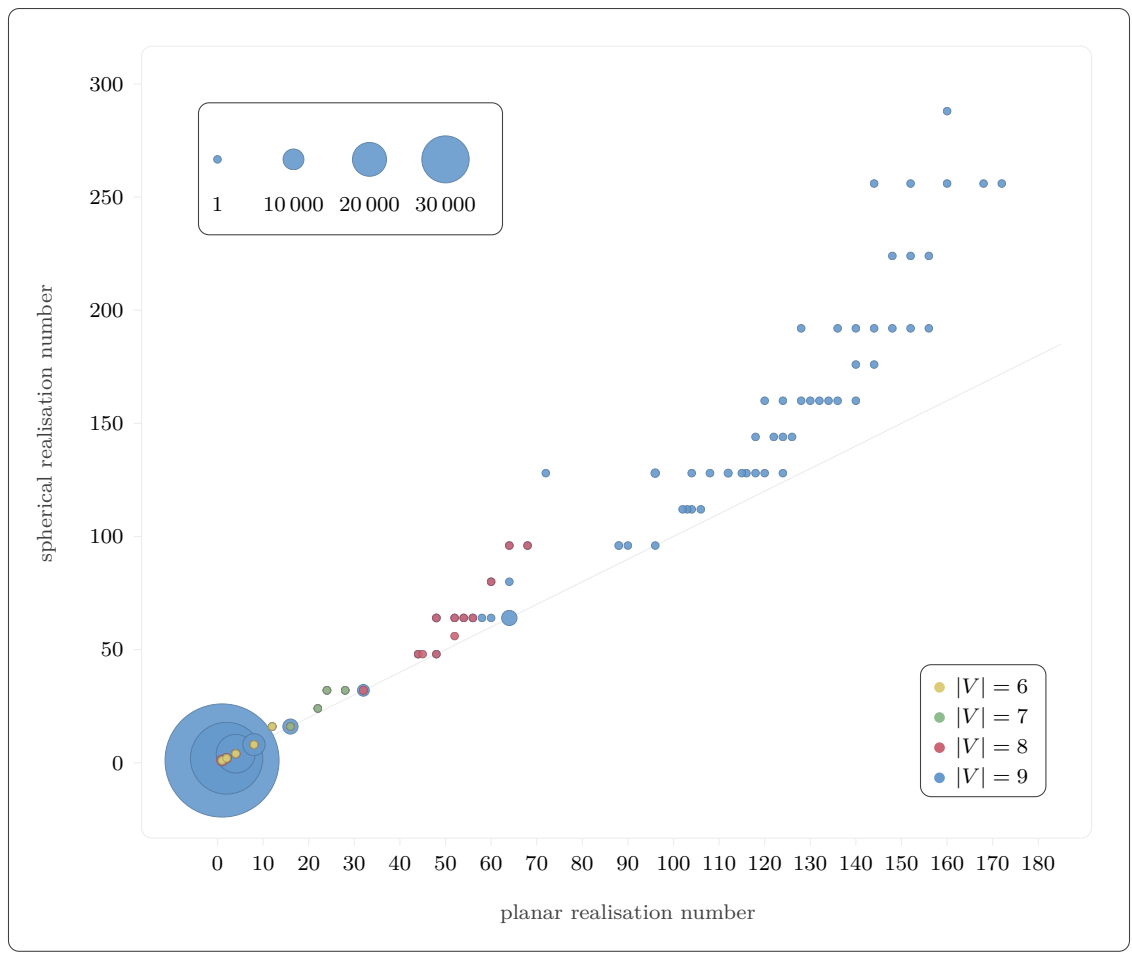
\begin{figure}[ht]
    \centering
    \begin{tikzpicture}
		\draw[cola,rounded corners] (-2.75,-2.5) rectangle (12,10);
		\node[alabelsty,rotate=90] at ($0.03*(0,150)+(-2.25,0)$) {spherical realisation number};
		\node[alabelsty] at ($0.06*(90,0)+(0,-2)$) {planar realisation number};
		\draw[bline,rounded corners] (-1,-1) rectangle (11.5,9.5);
		\foreach \yy [evaluate=\yy as \y using \yy*0.03] in {0,50,100,...,300} \draw[bline] (-1,\y) -- +(-0.1,0) node[left,labelsty] {$\yy$};
		\foreach \xx [evaluate=\xx as \x using \xx*0.06] in {0,10,...,180} \draw[bline] (\x,-1) -- +(0,-0.1) node[below,labelsty] {$\xx$};

		\draw[bline] (0,0)--($0.06*(185,0)+0.03*(0,185)$);
		\foreach \list [count=\i] in {{1/1/1.5,2/2/0.951937,4/4/0.512383,8/8/0.293102,64/64/0.20139,16/16/0.199893,32/32/0.15516,96/128/0.112254,112/128/0.104874,48/64/0.103533,88/96/0.103325,24/32/0.102889,12/16/0.101654,56/64/0.100731,136/192/0.100453,108/128/0.100418,44/48/0.1004,128/192/0.100383,120/160/0.100348,28/32/0.100348,104/128/0.100313,96/96/0.100226,22/24/0.100191,104/112/0.100122,128/160/0.100104,120/128/0.100104,144/192/0.100087,116/128/0.100087,90/96/0.100087,144/256/0.10007,156/192/0.100035,152/192/0.100035,124/160/0.100035,140/192/0.100017,140/160/0.100017,136/160/0.100017,132/160/0.100017,130/160/0.100017,126/144/0.100017,118/128/0.100017,103/112/0.100017,68/96/0.100017,64/96/0.100017,172/256/0.1,168/256/0.1,160/288/0.1,160/256/0.1,156/224/0.1,152/256/0.1,152/224/0.1,148/224/0.1,148/192/0.1,144/176/0.1,140/176/0.1,134/160/0.1,124/144/0.1,124/128/0.1,122/144/0.1,118/144/0.1,115/128/0.1,106/112/0.1,102/112/0.1,72/128/0.1,64/80/0.1,60/80/0.1,60/64/0.1,58/64/0.1,54/64/0.1,52/64/0.1,48/48/0.1},{1/1/0.140817,2/2/0.12731,4/4/0.115387,32/32/0.109121,8/8/0.108616,16/16/0.104717,48/64/0.100818,56/64/0.100244,44/48/0.100191,24/32/0.100174,12/16/0.100122,28/32/0.100017,68/96/0.1,64/96/0.1,60/80/0.1,54/64/0.1,52/64/0.1,52/56/0.1,48/48/0.1,45/48/0.1,22/24/0.1},{1/1/0.10228,2/2/0.10148,16/16/0.101097,4/4/0.100975,8/8/0.100522,24/32/0.100052,28/32/0.1,22/24/0.1,12/16/0.1},{1/1/0.100244,8/8/0.100191,2/2/0.100139,4/4/0.10007,12/16/0.1}}
		{
			\foreach \xx/\yy/\si [evaluate=\xx as \x using \xx*0.06,evaluate=\yy as \y using \yy*0.03] in \list
			{
				\node[dvertex,draw=col\i!80!black,fill=col\i,minimum size=\si cm,opacity=0.9] at (\x,\y) {};
			}
		}

		\begin{scope}
			\coordinate (s) at (0,-0.4);
			\coordinate (l) at ($(9.5,1)$);
			\draw[cola,fill=white,rounded corners] ($(l)+(-0.25,0.3)$) rectangle ++(1.65,-1.75);
			\node[vertex,col4,legend,label={[labelsty]0:$|V|=6$}] at (l) {};
			\node[vertex,col3,legend,label={[labelsty]0:$|V|=7$}] at ($(l)+(s)$) {};
			\node[vertex,col2,legend,label={[labelsty]0:$|V|=8$}] at ($(l)+2*(s)$) {};
			\node[vertex,col1,legend,label={[labelsty]0:$|V|=9$}] at ($(l)+3*(s)$) {};
		\end{scope}
		\begin{scope}[xshift=-1cm,yshift=8cm]
			\draw[cola,fill=white,rounded corners] (0.75,-1) rectangle ++(4,1.75);
			\foreach \s/\t [count=\i] in {0.1/1, 0.274043/10000, 0.448103/20000, 0.622163/30000}
			{
				\node[dvertex,draw=col1!80!black,fill=col1,minimum size=\s cm,opacity=0.9] at (\i,0) {};
				\draw[bline] (\i,-0.5)--++(0,-0.1) node[labelsty] {\pgfmathprintnumber{\t}};
			}
		\end{scope}
	\end{tikzpicture}
    \caption{Spherical realisation numbers compared to the planar one. The size of the circles shows the amount of graphs with the respective numbers.}
    \label{fig:comparison}
\end{figure}

\addcontentsline{toc}{section}{Acknowledgments}
\section*{Acknowledgements}
This research was funded in part by the Austrian Science Fund (FWF)
10.55776/I6233. For open access purposes, the authors have applied a CC BY public copyright license to any author accepted manuscript version arising from this submission.

S.D.\ was supported by was supported by the Heilbronn Institute for Mathematical Research.
G.G.\ was partially supported by the Austrian Science Fund (FWF): 10.55776/P31888.
A.W.\ is partially supported by the Austrian Science Fund (FWF): 10.55776/PAT2559123.

\bibliography{rna}
\bibliographystyle{alphaurl}

\appendix

\section{The degree of a graph map is the realisation number of a graph}\label{app}

In this section we prove the following equality.

\begin{proposition}\label{prop}
    If a graph $G$ with at least two vertices is rigid then $\deg(p_G) = c_2(G)$.
\end{proposition}

For our proof, we require the following lemma.

\begin{lemma}\label{lem:strcon}
    Let $G$ be a rigid graph and $\lambda \in I_G$ be a generic point.
    \begin{enumerate}[label=(\roman*)]
        \item \label{item:strcon1} Each connected component of $f^{-1}_G(\lambda)$ contains a generic realisation.
        \item \label{item:strcon2} If $\rho_1,\rho_2 \in f^{-1}_G(\lambda)$ satisfy $f_{K_{V(G)}}(\rho_1) = f_{K_{V(G)}}(\rho_2)$,
        then there exists an affine transformation $h \in E_2$ such that $h(\rho_1) = h(\rho_2)$.
    \end{enumerate}
\end{lemma}

\begin{proof}
    \ref{item:strcon1} follows from a using a slight adaptation of \cite[Lemma 3.2]{JacksonOwen2019} and \ref{item:strcon2} follows \cite[Lemmas 3.1 \& 3.4]{JacksonOwen2019}.
\end{proof}

\begin{proof}[Proof of \Cref{prop}]
    First, we observe the following factorisation of the measurement map into dominant maps:
    \begin{center}       
        \begin{tikzpicture}
            \node[] (r)at (0,0) {$\mathbb C^{2|V(G)|}$};
            \node[] (p) at (4,0) {$R_{V(G)}$};
            \node[] (c) at (4,-2) {$I_G$};
            \draw[cdarrow] (r) to node[above,labelsty] {$f_{K_{V(G)}}$} (p);
            \draw[cdarrow] (r) to node[below,labelsty] {$f_G$} (c);
            \draw[cdarrow] (p) to node[right,labelsty] {$p_G$} (c);
        \end{tikzpicture}
    \end{center}
    Since $G$ is rigid, $p_G$ is generically finite.
    Choose a generic point $\lambda \in I_G$.
    By \Cref{lem:strcon}\ref{item:strcon1}, this is equivalent to choosing a generic realisation $\rho \in \mathbb{C}^{2|V(G)|}$ and fixing $\lambda = f_G(\rho)$.
    Fix $p_G^{-1}(\lambda) = \{\mu_1, \ldots , \mu_s\}$ and fix $C_{ 1},\ldots,C_{t}$ to be the connected components of $f^{-1}_G(\lambda)$.    
    By \Cref{lem:strcon}\ref{item:strcon1} and \cite[Theorem 3.6]{JacksonOwen2019}, $t= 2c_2(G)$, with each connected component consisting of the orbit of a single realisation under the orientation-preserving\footnote{Every affine map $E_2$ is of the form $z \mapsto T(z) + z_0$ for some vector $z_0 \in \mathbb{C}^2$ and some linear $2 \times 2$ matrix $T$ where $T^T T = TT^T = I_2$. Either $\det(T) = 1$ and $T$ is orientation-preserving, or $\det(T) = -1$ and $T$ is not orientation-preserving.} affine transformations of $E_2$;
    moreover, if we set $c=c_2(G)$, we can relabel our connected components $C_{\pm 1}, \ldots, C_{\pm c}$ so that $C_{-i}$ is the image of $C_i$ under the reflection $(x,y) \mapsto (x,-y)$.

    For each $i \in \{\pm 1, \ldots, \pm c\}$,
    set $\sigma(i) \in \{1,\ldots,s\}$ to be the index such that $f_{K_{V(G)}}(C_i) = \{\mu_{\sigma(i)}\}$.
    It is immediate that $\sigma(-i) = \sigma(i)$.
    As each map in the above diagram is dominant and $p_G$ is generically finite,
    the genericity of $\lambda$ implies that the map $\sigma$ is surjective (and so $s \leq 2c$).
    Suppose that there exists $i,j \in \{1,\ldots,c\}$ such that $\sigma(i) = \sigma(j)$.
    This implies the existence of realisations $\rho_i,\rho_j \in f_G^{-1}(\lambda)$ where $f_{K_{V(G)}}(\rho_i) = f_{K_{V(G)}}(\rho_j)$ but such that $h(\rho_i) \neq h(\rho_j)$ for each $h \in E_2$.
    However, this now directly contradicts \Cref{lem:strcon}\ref{item:strcon2}.    
    Hence, $s  = 2c$, which implies the desired result.
\end{proof}

\section{Graph encodings}\label{sec:cert}

Here a graph is encoded as in \cite{GraseggerKoutschanTsigaridas2020} by an integer,
derived from the upper triangular part of adjacency matrix read row-wise as a binary number.
For instance the triangle graph is denoted by $(1,1,1)_2=7$ and $K_4$ minus on edge as $(011111)_2=31$.
\pyrigi\ \cite{PyRigi} contains python methods to decode and encode graphs.

In the following (\Cref{tab:max_cert,tab:max_cert_sphere}) we provide graphs which obtain the realisations numbers that are used in computational results (\Cref{sec:compute}).

\begin{table}[ht]
    \centering\small
    \begin{tabular}{rrrr}
        \toprule
        $|V|$ & $k=1$ & $k=2$ & $k=3$ \\\midrule
         6 & 3327 & 3583 & 4095 \\
         7 & 1624383 & 101887 & 102399 \\
         8 & 155852367 & 210799359 & 204542975 \\
         9 & 9548896180 & 45234555391 & 43630233599 \\
        10 & 20347466531983 & 17801747326540 & 6709897659391 \\
        11 & 19423424626348167 & 739685790686724 & 2626220166634959 \\
        12 & 9601886131857279073  \\\bottomrule
    \end{tabular}
    \caption{Certificate graphs for the realisation numbers in \Cref{tab:maxk}.}
    \label{tab:max_cert}
\end{table}

The graph $G$ with $\ccount{G}=172$ which is the maximal value for graphs with 11 vertices, $2|V|-1$ edges and a globally rigid subgraph $H$ with $|E(H)| = 2|V(H)|-1$, has integer encoding 23084260116373631.

\begin{table}[ht]
    \centering\small
    \begin{tabular}{rrrr}
        \toprule
        $|V|$ & $k=1$ & $k=2$ & $k=3$ \\\midrule
         6 & 3327 & 3583 & 4095 \\
         7 & 1624383 & 101887 & 102399 \\
         8 & 7156974 & 210799359 & 204542975 \\
         9 & 9548896180 & 975773247 & 1009849343 \\
        10 & 4778694408096 & 6086548036671 & 6709897659391  \\
        11 & 5916760438521919 & 18190583547111768 & 10213445215953215 \\\bottomrule
    \end{tabular}
    \caption{Certificate graphs for the realisation numbers in \Cref{tab:maxk_sphere}.}
    \label{tab:max_cert_sphere}
\end{table}

\end{document}